\documentclass[10pt,oneside,a4paper]{amsart}
\usepackage{amsmath,amsfonts,amsthm, amssymb,textcomp}
\usepackage{array}
\usepackage[all]{xy}
\usepackage{portland}

\theoremstyle{plain}
\newtheorem{theorem}{Theorem}[section]

\newtheorem{lemma}[theorem]{Lemma}

\theoremstyle{definition}

\newcommand{\C}{\mathbb{C}}

\newcommand{\Z}{\mathbb{Z}}

\newcommand{\ab}{\text{ab}}
\newcommand{\PP}{\mathbb{P}}
\newcommand{\OO}{\ensuremath{\mathcal{O}}}

\CompileMatrices
\SelectTips{cm}{11}

\title{Homology of some surfaces with $p_g = q = 0$ isogenous to a product}
\author{Timofey Shabalin}
\address{State University Higher School of Economics, Department of Mathematics, 20 Myasnitskaya st, Moscow 101000, Russia} 
\email{shabalin.timofey@gmail.com}

\begin{document}
\maketitle

\begin{abstract}
Bauer and Catanese \cite{bauercat} have found 4 families of surfaces of general type with $p_g = q = 0$ which are quotients of the product of curves by the action of finite abelian group. We compute integral homology groups of these surfaces.
\end{abstract}

\section{Introduction}

There is a complete classification of surfaces of special type, that is, with Kodaira dimension $< 2$. But the question of classification of surfaces of general type is still open. The most interesting and complicated class of these surfaces is formed by surfaces with $p_g = q = 0$. Interest to these surfaces arose when Max Noether asked, whether a surface with $p_g = q = 0$ must be rational. First counterexample was constructed by Enriques. Later, counterexamples of general type were constructed by Godeaux and Campedelli.

Bauer and Catanese have worked on the classification of surfaces of general type with $p_g = q = 0$ which are covered by a product of curves. In \cite{bauercat} they give a complete list of such surfaces which are quotients by diagonal action of a finite abelian group on the product of curves. The question is reduced to the problem of finding all finite abelian groups which admit a system of generators with some properties. Such surfaces form 4 irreducible families. Derived categories of these surfaces were studied in \cite{GS}, \cite{lee}.

The main result of the paper is the computation of the homology groups with integral coefficients of all surfaces from these families. In the next section we recall some basic facts about such surfaces. The last section is devoted to the computation of homology groups.

I would like to thank Sergey Galkin, Yakov Kononov, Victor Kulikov, Evgeny Shinder for helpful discussions. I am grateful to my advisor Dmitri Orlov for his patience, encouragement and many valuable suggestions.

\section{Surfaces with $p_g = q = 0$ isogenous to a higher product with abelian group}\label{sec.1}

We consider smooth algebraic surfaces over $\C$. The geometric genus of a surface $S$ is the number $p_g = h^0(S,\Omega^2_S) = \dim H^0(S, \Omega^2_S)$. The irregularity of $S$ is $q = h^0(S, \Omega^1_S)$. By Hodge theory we have $p_g = h^2(S,\OO_S)$, $q = h^1(S,\OO_S)$. The Kodaira dimension $\kappa$ of $S$ is the degree of growth of $h^0(S, nK_S)$ in $n$, where $K_S$ is the canonical class. A surface is said to be of general type if $\kappa = 2$. 

One way of constructing surfaces of general type with $p_g = q = 0$ is by taking quotients by actions of finite groups. A surface $S$ is said to be isogenous to a higher product \cite{catanese} if it admits a finite unramified covering of the form $C_1 \times C_2$, where $C_1, C_2$ are curves of genera $\ge 2$. By \cite{catanese} it implies that $S$ is the quotient $(C_1 \times C_2)/G$ where $g(C_i) \ge 2$, $G$ is a finite group acting freely on $C_1 \times C_2$. Action is said to be of mixed type if $G$ exchanges the two factors, and of unmixed type if $G$ acts via a product action.

From now on $S$ will be the quotient $(C_1\times C_2)/G$, where the finite group $G$ acts on each curve $C_i$ of genus $\ge 2$ and the action on the product is free. We will also assume that $S$ has $p_g = q = 0$. The surface $S$ is a ramified covering of $C_1/G \times C_2/G$ and if there were non-zero differential forms on $C_1/G\times C_2/G$ we could pull them back to $S$. Thus the assumption $p_g = q = 0$ implies that $C_i/G$ is isomorphic to $\PP^1$.

Consider a ramified covering $p\colon C \to C/G \cong \PP^1$. Let $B$ be the branching locus of $p$: $B = \{p(x)\mid x\in C \text{ is stabilized by a non-trivial subgroup of } G\}$. Then we have an exact sequence

$$
1 \to \pi_1(C\setminus p^{-1}(B)) \to \pi_1(\PP^1\setminus B) \to G \to 1.
$$

Let $x_1, \dots, x_r$ be the distinct points of $B$, $\gamma_i$ be a simple geometric loop around~$x_i$. Each $\gamma_i$ is mapped to an element $g_i \in G$. Denote the order of~$g_i$ by~$m_i$. The orbifold fundamental group of the covering $p$ is defined to be the quotient of $\pi_1(\PP^1\setminus B)$ by the normal subgroup generated by elements $\gamma_i^{m_i}$.

In our situation, let $\Pi(i)$ for $i=1,2$ be the orbifold fundamental group of the covering $C_i \to C_i/G$. There is an exact sequence

$$
1 \to \pi_1(C_i) \to \Pi(i) \to G \to 1.
$$

Multiplying the two sequences we get

$$
1 \to \pi_1(C_1) \times \pi_2(C_2) \to \Pi(1) \times \Pi(2) \to G \times G \to 1.
$$

Then $\pi_1(S)$ is the preimage in $\Pi(1) \times \Pi(2)$ of the diagonal subgroup $G$ in $G\times G$ \cite{catanese}. Thus for abelian $G$ we have an exact sequence

\begin{equation}\label{seq}
1 \to \pi_1(S) \to \Pi(1) \times \Pi(2) \to G \to 1,
\end{equation}

where the right map is the composition of $\Pi(1) \times \Pi(2) \to G \times G$ with the map $G \times G \to G$ given by $(a,b) \mapsto a - b$.

The covering $C \to \PP^1$ is determined by the choice of branching points $x_1,\dots,x_r$ on $\PP^1$ and images $g_i\in G$ of loops around $x_i$ such that $g_1\dots g_r = 1$ and $g_1,\dots, g_r$ generate $G$. If $\{g_i\}$ is the system of generators corresponding to the covering $C_1~\to~\PP^1$, $\{h_j\}$ corresponds to $C_2 \to \PP^1$ then the action of $G$ on $C_1 \times C_2$ is free if and only if we have 

\begin{equation} \notag
\left(\bigcup \langle g_i \rangle \right)\cap \left(\bigcup \langle h_j \rangle \right) = \{1\},
\end{equation}

\noindent
where $\langle g_i \rangle$ is the subgroup generated by $g_i$.

Bauer and Catanese classified surfaces with $p_g = q = 0$ isogenous to a higher product with abelian group $G$ \cite{bauercat}.
They found all abelian groups $G$ admitting two systems of generators $\{g_i\}$, $\{h_i\}$ as above. It turned out that the only possible groups $G$ are $(\Z/2\Z)^3$, $(\Z/2\Z)^4$, $(\Z/3\Z)^2$, $(\Z/5\Z)^2$. Varying of the positions of branching points on the quotients $C_1/G$, $C_2/G$ corresponds to varying of a surface $S$ in a family. Bauer and Catanese found that these families have dimensions $5, 4, 2$ for groups $(\Z/2\Z)^3$, $(\Z/2\Z)^4$, $(\Z/3\Z)^2$ respectively and consist of two points for $(\Z/5\Z)^2$. In the next section we will find the homology groups of these surfaces.

\section{Computation of homology groups}

We want to find the first homology group with integral coefficients of each surface $S$ from one of the 4 families described in \cite{bauercat}. All these surfaces are quotients $(C_1\times C_2)/G$ of a product of two curves $C_i$ by the product action of a finite abelian group $G$. Both curves are ramified coverings of $C_i/G \cong \PP^1$. Since we have $H_1(S,\Z) \cong \pi_1(S)^{\ab} := \pi_1(S)/[\pi_1(S),\pi_1(S)]$, the problem is reduced to a group-theoretic computation: given two homomorphisms $\Pi(i) \to G$, find the abelianization of the kernel in the sequence \eqref{seq}.

The group $\Pi(1)$ has the form $\langle a_1,\dots,a_n \mid a_1^{k_1}, \dots, a_n^{k_n}, a_1\cdots a_n \rangle$, where $n$ is the number of points in the branching locus $B$ of the covering $C_1 \to \PP^1$, $a_j$ is the class of a loop around the point $x_j \in B$ and $k_j$ is the order of the stabilizer of each preimage of $x_j$. The number $k_j$ is equal to the order of the image of $a_j$ in the group~$G$. In all our cases $G$ will be the product of cyclic groups of the same order, so all $k_j$ will be equal to some $k$. The group $\Pi(2)$ is of the same form, with generators denoted by $b_1, \dots, b_m$. We will denote the generators of the cyclic factors of $G$ by $e_i$. Now we will describe the homomorphisms $\phi: \Pi(1) \to G$ and $\psi: \Pi(2) \to G$, which are the same for all surfaces in each of the 4 families.

Case 1: $G = (\Z/2\Z)^3$. We have $k = 2$, $n=5$, $m=6$ and the homomorphism $\phi$ is given by 
$$
a_1 \mapsto e_1, \;
a_2 \mapsto e_2, \;
a_3 \mapsto e_3, \;
a_4 \mapsto e_1, \;
a_5 \mapsto e_2 + e_3.
$$
The homomorphism $\psi$ is
\begin{align*}
b_1 \mapsto e_1+e_2, \;
b_2 \mapsto e_1+e_3, \;
b_3 \mapsto e_1+e_2+e_3, \\
b_4 \mapsto e_1+e_2, \;
b_5 \mapsto e_1+e_3, \;
b_6 \mapsto e_1+e_2+e_3.
\end{align*}

Case 2: $G = (\Z/2\Z)^4$. We have $k = 2$, $n=m=5$. The homomorphism $\phi$ is given by
$$
a_1 \mapsto e_1, \;
a_2 \mapsto e_2, \;
a_3 \mapsto e_3, \;
a_4 \mapsto e_4, \;
a_5 \mapsto e_1+e_2+e_3+e_4
$$
and $\psi$ by
$$
b_1 \mapsto e_2+e_3+e_4, \;
b_2 \mapsto e_1+e_3+e_4, \;
b_3 \mapsto e_1+e_3, \;
b_4 \mapsto e_2+e_4, \;
b_5 \mapsto e_3+e_4.
$$

Case 3: $G = (\Z/3\Z)^2$. We have $k = 3$, $n=m=4$. The homomorphism $\phi$ is given by
$$
a_1 \mapsto e_1, \;
a_2 \mapsto e_2, \;
a_3 \mapsto -e_1, \;
a_4 \mapsto -e_2
$$
and $\psi$ by
$$
b_1 \mapsto e_1+e_2, \;
b_2 \mapsto e_1-e_2, \;
b_3 \mapsto -e_1-e_2, \;
b_4 \mapsto -e_1+e_2.
$$

Case 4: $G = (\Z/5\Z)^2$. We have $k = 5$, $n=m=3$. The homomorphism $\phi$ is given by
$$
a_1 \mapsto e_1, \;
a_2 \mapsto e_2, \;
a_3 \mapsto -e_1-e_2
$$
and $\psi$ by
$$
b_1 \mapsto e_1+2e_2, \;
b_2 \mapsto 3e_1+4e_2, \;
b_3 \mapsto e_1+4e_2.
$$

We introduce the following notation: $K = \pi_1(S)$, $F = \Pi(1) \times \Pi(2)$. We have an exact sequence

$$
1 \to K \to F \to G \to 1.
$$

In induces an exact sequence

\begin{equation} \label{seq2}
1 \to [F,F]/[K,K] \to F/[K,K] \to F/[F,F] \to 1.
\end{equation}

Later we will prove that the group $[F,F]/[K,K]$ is in fact abelian. We will first compute the groups $F/[F,F]$, $[F,F]/[K,K]$ and the 2-cocycle corresponding to the extension \eqref{seq2}.  Then the extension 

$$
1 \to [F,F]/[K,K] \to K/[K,K] \to K/[F,F] \to 1
$$
will be given by the restriction of the 2-cocycle to the group $K/[F,F]$.

We use the following convention for the commutator of two elements: $[a,b]=aba^{-1}b^{-1}$. The commutator satisfies two relations

\begin{align}\label{ident1}
[ab,c] = a[b,c]a^{-1}[a,c], \\ \label{ident2}
[a,bc] = [a,b]b[a,c]b^{-1}. 
\end{align}

\begin{lemma}
We have $[F,[F,F]]\subseteq [K,K]$.
\end{lemma}

\begin{proof}
The group $[F,[F,F]]$ is generated by elements of the form $[p_1q_1,[p_2q_2,p_3q_3]]$, where $p_i$ are elements of $\Pi(1)$, considered as the subgroup of $F$ and $q_i \in \Pi(2)$. We have $[p_1q_1,[p_2q_2,p_3q_3]] = [p_1,[p_2,p_3]][q_1,[q_2,q_3]]$. Since $G$ is abelian, $[p_2,p_3],[q_2,q_3]\in K$. The homomorphisms $\phi,\psi$ are surjective, so there exist $p_1'\in \Pi(1), q_1'\in \Pi(2)$ such that $p_1'q_1, p_1q_1' \in K$. Then $[p_1,[p_2,p_3]][q_1,[q_2,q_3]]=[p_1q_1',[p_2,p_3]][p_1'q_1,[q_2,q_3]]\in [K,K]$.
\end{proof}

Therefore, the group $[F,F]/[K,K]$ is abelian. The identities (\ref{ident1},~\ref{ident2}) imply that the commutator map
$$
[\cdot,\cdot]\colon F\times F \to [F,F]/[K,K]
$$
is bilinear and skew-symmetric. It induces the surjective homomorphism
$$
\tau\colon \Lambda^2 F^{\ab} \to [F,F]/[K,K],
$$
taking $x \wedge y$ to the class of $[x,y]$. Here by $\Lambda^2 F^{\ab}$ we mean the quotient of the tensor product $F^{\ab} \otimes F^{\ab}$, taken in the category of abelian groups, by the subgroup generated by elements of the form $x \otimes x$.

Suppose $p_1 \in K \cap \Pi(1)$, $p_2q_2$ an arbitrary element of $F$. There exists $q_2'\in \Pi(2)$ such that $p_2q_2'\in K$. Then $\tau(p_1 \wedge p_2q_2) = \tau(p_1 \wedge p_2) = \tau (p_1 \wedge p_2q_2') = 0$. Analogously, if $q_1 \in K \cap \Pi(2)$, then $\tau(q_1 \wedge p_2q_2) = 0$. Let us define 
$$
\sigma\colon \Lambda^2G \to [F,F]/[K,K]
$$
by the rule $\sigma(g_1 \wedge g_2) = \tau(p_1 \wedge p_2)$, where $p_i \in \Pi(1)$ are arbitrary elements such that $\phi(p_i) = g_i$. It is well-defined by the above. If $p_1q_1, p_2q_2 \in K$, then we have $\phi(p_i) = -\psi(q_i)$ and $\tau(p_1q_1 \wedge p_2q_2) = 0$, so $\tau(p_1 \wedge p_2) = -\tau(q_1 \wedge q_2)$. This means that for arbitrary $p,p'\in \Pi(1)$ and $q,q'\in \Pi(2)$ we have $\tau(p\wedge p') = \sigma(\phi(p) \wedge \phi(p'))$ and $\tau(q\wedge q') = -\sigma(\psi(q) \wedge \psi(q'))$. Thus the image of $\tau$ is generated by the image of $\sigma$ and the homomorphism $\sigma$ is surjective.

Applying the identities (\ref{ident1},~\ref{ident2}), we find that in the group $[F,F]/[F,[F,F]]$ we have
\begin{align*}
1 = a_1^k\cdots a_n^k = a_1^k \cdots a_{n-1}^k(a_{n-1}^{-1}\cdots a_1^{-1})^k = \\
 \frac{k(k-1)}2 \prod_{i<j<n}[a_i^{-1},a_j^{-1}] = \frac{k(k-1)}2 \prod_{i<j<n}[a_i,a_j].
\end{align*}

This means that the elements $\frac{k(k-1)}2 \sum_{i<j<n}\phi(a_i)\wedge\phi(a_j)$, $\frac{k(k-1)}2 \sum_{i<j<m}\psi(b_i)\wedge\psi(b_j)$ lie in the kernel of $\sigma$. We will denote by $H$ the quotient of $\Lambda^2 G$ by the subgroup generated by these two elements and, by abuse of notation, by $\sigma$ the induced homomorphism $\sigma\colon H \to [F,F]/[K,K]$.

\begin{lemma}
The homomorphism $\sigma\colon H \to [F,F]/[K,K]$ is an isomorphism.
\end{lemma}

\begin{proof}
We will describe the inverse homomorphism $\alpha\colon [F,F]/[K,K] \to H$. We will use all generators of the groups $\Pi(1)$, $\Pi(2)$, except the last one. Let $\Sigma_1$ be the set $\{a_1^{\pm1},\dots,a_{n-1}^{\pm1}\}$ and $\Sigma_2 = \{b_1^{\pm1},\dots,b_{m-1}^{\pm1}\}$. Each element of $[F,F] = [\Pi(1),\Pi(1)] \times [\Pi(2),\Pi(2)]$ can be represented by a pair of words $(u,v)$ in the letters $\Sigma_1$, $\Sigma_2$ respectively, such that the total degree of $u$ in each $a_i$ is divisible by $k$, and the same is true about the degrees of $v$ in $b_j$'s. Suppose $(u,v) = (x_1\dots x_c, y_1 \dots y_d)$, where $x_l = a_{i(l)}^{\pm 1}$, $y_l = b_{j(l)}^{\pm 1}$. We define
$$
\alpha(u,v) = \sum_{\substack{r<s \\ i(r) > i(s)}} \phi(x_r) \wedge \phi(x_s) - \sum_{\substack{r<s \\ j(r) > j(s)}} \psi(y_r) \wedge \psi(y_s).
$$
If we add $xx^{-1}$ or $x^k$ for $x\in \Sigma_1$ in an arbitrary place in the word $u$, the image $\alpha(u,v)$ will not change. If we add $(a_1\dots a_{n-1})^k$, the value will change by $\frac{k(k-1)}2 \sum_{i>j} \phi(a_i)\wedge\phi(a_j)$, which is zero in $H$. Thus we have shown that $\alpha(u,v)$ does not depend on the choice of the word $u$ representing the given element of $[\Pi(1),\Pi(1)]$. The same is true about $v$.

For the concatenation of two pairs of words we have
$$
\alpha(uu',vv') = \alpha(u,v) + \alpha(u',v') + R,
$$
where $R$ is the sum
$$
\sum \phi(a_i^{\pm 1}) \wedge \phi(a_j^{\pm 1}) - \sum \psi(b_r^{\pm 1}) \wedge \psi(b_s^{\pm 1})
$$
over all occurrencies of letters with index $i$ in $u$, with $j$ in $u'$, respectively with index $r$ in $v$, with $s$ in $v'$ and $i > j$, $r > s$. Since the total degrees of $u, u', v, v'$ in each generator are divisible by $k$, the image of $R$ in $H$ is zero. Therefore we have shown that $\alpha$ is a well defined homomorphism $[F,F] \to H$.

If $(u,v)$ is a pair of words representing an element of $[F,F]$, then
$$
\alpha(a_i u a_i^{-1}, v) = \alpha(u,v) + R,
$$
where $R$ is the sum
$$
\sum_{\substack{x_j \text{ from } u \\ i > j}}\phi(a_i)\wedge \phi(x_j) + \sum_{\substack{x_j \text{ from } u \\ j > i}}\phi(x_j)\wedge \phi(a_i^{-1}) = \phi(a_i) \wedge \phi(u) = 0.
$$
The same applies to conjugation by generators of $\Pi(2)$. We have shown that $\alpha(xyx^{-1}) = \alpha(y)$ for arbitrary $x \in F$, $y \in [F,F]$. Thus $\alpha$ is a homomorphism $[F,F]/[F,[F,F]] \to H$. 

From the definition of $\alpha$ it follows that $\alpha([a_i,a_j])=\phi(a_i) \wedge \phi(a_j)$ and $\alpha([b_i,b_j])= -\psi(b_i) \wedge \psi(b_j)$. By identities (\ref{ident1},~\ref{ident2}), this means that for arbitrary $p,p'\in \Pi(1)$ and $q,q'\in \Pi(2)$ we have $\alpha([p,p'])=\phi(p) \wedge \phi(p')$ and $\alpha([q,q'])= -\psi(q) \wedge \psi(q')$. Now it is clear that $\alpha$ gives the homomorphism $[F,F]/[K,K] \to H$ inverse to $\sigma$.
\end{proof}

Recall \cite{weibel} that extensions of an arbitrary group $G$ by abelian group $A$, equipped with the structure of a $G$-module,
\begin{equation}\label{exten}
0 \to A \to E \to G \to 1,
\end{equation}
are parameterized by elements of the cohomology group $H^2(G,A)$. Specifically, to an extension of the form \eqref{exten} we can assign a normalized 2-cocycle $\langle\cdot,\cdot\rangle \colon G\times G \to A$, given by $\langle g,h \rangle = \gamma(g)\gamma(h)\gamma(gh)^{-1}$, where $\gamma$ is an arbitrary section of the map $\pi\colon E \to G$ such that $\gamma(1) = 1$. By definition, a normalized 2-cocycle is a map $\langle\cdot,\cdot\rangle \colon G\times G \to A$ satisfying two conditions:
\begin{align*}
\langle g,1 \rangle = \langle 1,g \rangle = 0, \\
f\cdot \langle g,h \rangle - \langle fg,h \rangle + \langle f,gh \rangle - \langle f,g \rangle = 0.
\end{align*}
Given the 2-cocycle $\langle\cdot,\cdot \rangle$, the group $E$ can be reconstructed (up to isomorphism) as follows: let $E = A \times G$ as a set, with multiplication given by
$$
(a,g) \cdot (b,h) = (a + g\cdot b + \langle g,h \rangle, gh).
$$

We have the following lemma.

\begin{lemma}[cf. \cite{weibel}] \label{lemma}
Suppose the extension
$$
0 \to A \to E \to G \to 1
$$
is given by the 2-cocycle $\eta \in Z^2(G,A)$. Let $f\colon H \to G$ be a homomorphism of groups. Then the extension
$$
0 \to A \to E\times_G H \to H \to 1
$$
is given by the pullback $f^* \eta \in Z^2(H,A)$. Here the action of $H$ on $A$ is induced by the homomorphism $f$.
\end{lemma}

In our situation, consider the extension \eqref{seq2}: 
$$
0 \to H \to F/[K,K] \to F^{\ab} \to 1.
$$
Since $[F,[F,F]] \subseteq [K,K]$, the action of $F^{\ab}$ by conjugation on $H$ is trivial. The group $F^{\ab}$ is a free $\Z/k\Z$-module with basis $\{a_1,\dots,a_{n-1},b_1,\dots,b_{m-1}\}$. We define the section $\gamma$ by
$$
\gamma\left(\sum_{i=1}^{n-1} r_ia_i + \sum_{i=1}^{n-1} r_i' b_i \right) = a_1^{r_1}\cdots a_{n-1}^{r_{n-1}} b_1^{r_1'} \cdots b_{m-1}^{r_{m-1}'}
$$
where $0 \le r_i,r_i' \le k-1$ for all $i$. Then
\begin{align*}
\left\langle \sum r_ia_i + \sum r_i' b_i, \sum s_ia_i + \sum s_i' b_i \right\rangle = \\
\sigma(
a_1^{r_1}\cdots a_{n-1}^{r_{n-1}} b_1^{r_1'} \cdots b_{m-1}^{r_{m-1}'}
a_1^{s_1}\cdots a_{n-1}^{s_{n-1}} b_1^{s_1'} \cdots b_{m-1}^{s_{m-1}'}
b_{m-1}^{-t_{m-1}'}\cdots b_1^{-t_1'} a_{n-1}^{-t_{n-1}}\cdots a_1^{-t_1}
),
\end{align*}
where $t_i$ can be equal to $r_i+s_i$ or to $r_i+s_i-k$ and similarly for $t_i'$. From the definition of $\sigma$ it follows that this is equal to
$$
-\sum_{i < j} r_j s_i \phi(a_i) \wedge \phi(a_j) + \sum_{i < j} r_j' s_i' \psi(b_i) \wedge \psi(b_j).
$$
Thus, the 2-cocycle is bilinear in our case.

Suppose we have an extension in the category of abelian groups
$$
0 \to A \to E \to G \to 0
$$
and we are given a surjective homomorphism $\Z^n \to G$, $e_i \mapsto g_i$ with kernel generated by some elements $r_j, j\in J$. Then we can define a homomorphism $A \oplus \Z^n \to E$ to be identity on $A$ and to map $e_i$ to $(0,g_i)$. It is well-defined, since $E$ is abelian and it is easy to see that it is in fact surjective and its kernel is generated by elements $(a_j,r_j), j\in J$, where $a_j\in A$ is the unique element such that $(a_j,r_j)$ maps to zero in $E$.

We have an exact sequence of abelian groups
$$
0 \to H \to K^{\ab} \to K/[F,F] \to 0.
$$
The 2-cocycle corresponding to it is the restriction of the 2-cocycle $\langle \cdot, \cdot \rangle$ to the subgroup $K/[F,F] \le F^{\ab}$, which we will denote by $\langle \cdot, \cdot \rangle$ by abuse of notation. The group $K/[F,F]$ is a free $\Z/k\Z$-module (recall that in all our cases we have $k$ prime). Thus it is a quotient of some $\Z^l$ by relations $kc_i = 0$. Since the 2-cocycle $\langle \cdot, \cdot \rangle$ is bilinear we have in the group $K^{\ab}$
$$
(0,c_i)^k = (\langle c_i,c_i \rangle + \cdots + \langle (k-1)c_i, c_i \rangle, kc_i) = \left( \frac{k(k-1)}2 \langle c_i, c_i \rangle, 0 \right).
$$
Therefore, to obtain the group $K^{\ab}$ we have to add to the abelian group $H$ new generators $f_i$ with relations $kf_i = -\frac{k(k-1)} 2 \langle c_i, c_i \rangle$. Now we will do this case by case.

\begin{theorem} The first homology groups of our surfaces are the following. \\
1. In case $G=(\Z/2\Z)^3$, $H_1(S,\Z) \cong (\Z/2\Z)^4 \oplus (\Z/4\Z)^2$. \\
2. In case $G=(\Z/2\Z)^4$, $H_1(S,\Z) \cong (\Z/4\Z)^4$. \\
3. In case $G=(\Z/3\Z)^2$, $H_1(S,\Z) \cong (\Z/3\Z)^5$. \\
4. In case $G=(\Z/5\Z)^2$, $H_1(S,\Z) \cong (\Z/5\Z)^3$.
\end{theorem}

\begin{proof}
Case 1: The group $H$ is the quotient of $\Lambda^2 G$ by the subgroup generated by elements
$$
\sum_{i < j < 5}\phi(a_i)\wedge \phi(a_j) = e_2 \wedge e_3, \sum_{i < j < 6} \psi(b_i) \wedge \psi(b_j) = 0.
$$
That is, $H \cong (\Z/2\Z)^2$. We choose the basis of $K/[F,F]$ over $\Z/2\Z$: $c_1 = (a_1 + a_4, 0)$, $c_2 = (a_1 + a_2, b_1)$, $c_3 = (a_1 + a_3, b_2)$, $c_4 = (a_1 + a_2 + a_3, b_3)$, $c_5 = (a_1 + a_2, b_4)$, $c_6 = (a_1 + a_3, b_5)$. Then the group $K^{\ab}$ is the quotient of $H \oplus \Z^6$ by relations
\begin{align*}
2 f_1 & = -\langle c_1, c_1 \rangle = \phi(a_1) \wedge \phi(a_4) = 0, \\
2 f_2 & = e_1 \wedge e_2, \\
2 f_3 & = e_1 \wedge e_3, \\
2 f_4 & = e_1 \wedge (e_2+e_3) + e_2 \wedge e_3, \\
2 f_5 & = e_1 \wedge e_2, \\
2 f_6 & = e_1 \wedge e_3.
\end{align*}
The elements on the right hand side generate whole $H$. Therefore, $K^{\ab} \cong (\Z/2\Z)^4 \oplus (\Z/4\Z)^2$.

Case 2: The group $H$ is the quotient of $\Lambda^2 G$ by the subgroup generated by elements
\begin{align*}
\sum_{i < j < 4} \phi(a_i)\wedge \phi(a_j) & = e_1 \wedge (e_2+e_3+e_4) + e_2 \wedge (e_3+e_4) + e_3 \wedge e_4, \\
\sum_{i < j < 4} \psi(b_i) \wedge \psi(b_j) & = e_1 \wedge e_4 + e_2 \wedge e_3 + e_3 \wedge e_4.
\end{align*}
That is, $H \cong (\Z/2\Z)^4$. We choose the basis of $K/[F,F]$ over $\Z/2\Z$: $c_1 = (a_2 + a_3 + a_4, b_1)$, $c_2 = (a_1 + a_3 + a_4, b_2)$, $c_3 = (a_1 + a_3, b_3)$, $c_4 = (a_2 + a_4, b_4)$. Then the group $K^{\ab}$ is the quotient of $H \oplus \Z^4$ by relations
\begin{align*}
2 f_1 & = e_2 \wedge e_3 + e_2 \wedge e_4 + e_3 \wedge e_4, \\
2 f_2 & = e_1 \wedge e_3 + e_1 \wedge e_4 + e_3 \wedge e_4, \\
2 f_3 & = e_1 \wedge e_3, \\
2 f_4 & = e_2 \wedge e_4.
\end{align*}
The elements on the right hand side are linearly independent in $H$. Therefore, $K^{\ab} \cong (\Z/4\Z)^4$.

Cases 3 and 4: Since $k$ is odd, the number $\frac{k(k-1)}2$ is divisible by $k$. Thus, $H$ is the quotient of $\Lambda^2 G$ by trivial elements, $H = \Lambda^2 G$. And the relations on generators $f_i$ are just $kf_i = 0$. This means that $K^{\ab} \cong \Lambda^2 G \oplus K/[F,F]$. In case 3, $K^{\ab} \cong (\Z/3\Z)^5$. In case $4$, $K^{\ab} \cong (\Z/5\Z)^3$.
\end{proof}

\def\cprime{$'$} \def\cprime{$'$}
\providecommand{\bysame}{\leavevmode\hbox to3em{\hrulefill}\thinspace}
\providecommand{\MR}{\relax\ifhmode\unskip\space\fi MR }
\providecommand{\MRhref}[2]{%
  \href{http://www.ams.org/mathscinet-getitem?mr=#1}{#2}
}
\providecommand{\href}[2]{#2}

\end{document}